\documentclass[11pt]{article}
\usepackage[T1]{fontenc}
\usepackage{lmodern}
\usepackage{microtype}
\usepackage{xcolor}
\usepackage[utf8]{inputenc}
\usepackage[english]{babel}
\usepackage{amsmath,amsfonts,amssymb,amsthm,array}
\usepackage{graphicx}
\usepackage[colorinlistoftodos]{todonotes}
\usepackage{mathtools,tikz-cd}
\usepackage{cases}
\usepackage{tikz}
\usepackage{mathtools}
\usepackage{caption}
\usepackage{subcaption}
\usepackage{cellspace}
\usepackage[a4paper,
            left=3cm,
            right=3cm,
            top=3cm,
            bottom=3.5cm]{geometry}

\usepackage[
  colorlinks=true,
  linkcolor=blue,
  citecolor=blue,
  urlcolor=blue
]{hyperref}

\usepackage{hyperref}
\usepackage{breakurl} 
\usepackage{natbib}
\newcommand{\eps}{\varepsilon}

\renewcommand\phi{\varphi}

\numberwithin{equation}{section}

\newtheorem{theorem}{Theorem}[section]

\newtheorem{lemma}[theorem]{Lemma}

\theoremstyle{definition}

\newtheorem{remark}[theorem]{Remark}

\newtheorem{assm}[theorem]{Assumption}
\long\def\forget#1\forgotten{} 

\begin{document}
\title{Long-time asymptotics for multivariate Hawkes processes with long-range interactions} 
\date{}
\author{
Nadia Belmabrouk\footnote{
\begin{minipage}[t]{0.9\textwidth}
LAMA, UPEC \& CNRS, Université Paris-Est Créteil, 94010 Créteil, France.\\
Modal'X, Université Paris Nanterre, 92000 Nanterre, France.\\
\texttt{nadia.belmabrouk@parisnanterre.fr}
\end{minipage}
}
}

\maketitle

\begin{abstract}
We consider a system of interacting particles on an infinite graph, modeled by a multivariate Hawkes process with long-range interactions, where the interaction strength decays as a power law of the inter-particle distance with exponent $1+\alpha$.
This model is more intricate and realistic for some applications, such as neural networks, where long-range connections are present. Our main focus is to characterize the long-time behavior of the system depending on the range of the interactions. These results correspond to laws of large numbers. We prove that long-range interactions affect the limiting behavior in the subcritical case but not in the supercritical case.
The proofs of our results use properties of $\alpha$-stable laws, and Tauberian methods for Laplace transforms.
\end{abstract}



\maketitle

\section{Introduction}







The Hawkes process emerges as a powerful tool for modeling sequences of discrete events happening randomly and continuously in time, especially when each event's occurrence is influenced by all the previous events in the sequence. Introduced by A.G. Hawkes in 1971 \cite{hawkes1971spectra}, the univariate Hawkes process provides
a general framework for modeling time dependencies and self-excitation in point processes. The initial motivation came from modeling seismic activity \cite{hawkes1971spectra}, but the Hawkes process is now widely used in neuroscience (e.g., \cite{okatan2005analyzing}, \cite{reynaud2014goodness}), genomics \cite{reynaud2010adaptive}, order book modeling (e.g., \cite{lu2018high}, \cite{bauwens2009modelling}), and social interaction analysis \cite{zhou2013learning}.

In the context of interacting particle systems, this framework naturally extends to multivariate Hawkes processes, where each coordinate models the activity of a particle and interactions between particles are taken into account. For instance, in neuroscience, the activity of a neuron depends not only on its own past activity but also on the spiking activity of other neurons.
 In a multivariate Hawkes process, the dynamics of each particle $i$ of the system, is defined by a point process $(Z_t^i)_{t\geq 0}$ characterized by its stochastic intensity process $\big(\lambda_t^i \big)_{t \geq 0}$, satisfying for each $t \geq 0$:
\begin{equation}\label{intensity}
\lambda_t^i=\mathrm{P}(Z_t^i\ \text{has\ a\ jump\ in}\ [t,t+dt]\mid \mathcal{F}_t)= h_i\Big(\sum_{j\rightarrow i} \int_0^{t^-}\varphi_{ji}(t-u)dZ_u^j\Big),
\end{equation}
where $(\mathcal{F}_t)_{t\geq 0}$ is the natural history of the point processes, and the summation index~$j$ runs over all particles interacting with~$i$.

Here, $h_i$ is the {\em response function}, and $\varphi_{ji}$ is the {\em interaction function}  quantifying the effect of previous events from particle $j$ on the particle $i.$ Equivalently (see {\em e.g.} \cite{delattre2016hawkes}), the multivariate Hawkes process $Z=\big(Z_t^i\big)_{t\geq 0, i \in \mathbb{Z}}$ can be expressed as a solution to the system of stochastic differential equations driven for each $i$ by a Poisson measure $\pi^i$,
\begin{equation}\label{Hawkes}
Z_t^i=\int_0^t\int_0^{+\infty}\textbf{1}_{\{z\leq \lambda^i_s\}}\pi^i(dsdz),
\end{equation}
with the stochastic intensities $(\lambda^i_t)_{t\geq 0}$ defined by (\ref{intensity}).

Several models have been investigated; a typical example is the classical mean-field framework, in which particles are homogeneous and interact through a complete and finite graph, with identical interaction kernel mechanisms across all pairs. This framework has been extensively analyzed, addressing various questions related to large-population limits (when $N$ goes to infinity) \cite{erny2021conditional}, \cite{erny2022mean}, \cite{delattre2016hawkes}, and long-time behavior \cite{delattre2016hawkes}, \cite{bremaud1996stability}. More realistic frameworks have been proposed for various applications, including nearest-neighbor models (see \cite{delattre2016hawkes}) where the interaction graph on $\mathbb{Z}^d$ is not complete and each particle interacts only with its nearest neighbors. They study the long time behavior of the corresponding multivariate Hawkes process in the subcritical and supercritical cases.
Another class of models incorporates spatial dependencies, with interactions determined by the relative positions of the particles;
in \cite{agathe2022multivariate}, the author considers interactions that are governed by a position-dependent graph structure, while the strength of each interaction depends on the degree of the particle within that graph. The author establishes the large population limit of the system and investigates the long-time behavior of the limiting process in the subcritical regime. However, while long-range interactions are present in many physical and biological phenomena, these models feature a distance-independent interaction strength, and interactions between particles at arbitrarily large distances are not taken into account.

In this paper, we consider a different form of spatial dependency on the infinite complete graph over $\mathbb{Z}.$ Each $i\in\mathbb{Z}$ corresponds to the site of a particle, and the interaction between two particles~$i,j$ depends on their distance~$|i-j|$.
 More precisely,
 for any sites $i,j \in \mathbb{Z}$, in \eqref{intensity}, we consider:
\begin{itemize}
\item the {\em response function} as $h_i(x)=\mu_i +x$ where $\mu_i$ is the {\em "baseline intensity"} {associated with particle~$i$},
\item the {\em interaction function} $\varphi_{ji}$ as
\begin{equation}\label{eq:phialpha}
\varphi_{ji}(t) \propto \frac{\varphi(t)}{\vert i-j\vert^{\alpha+1}},
\end{equation}
for some parameter~$\alpha>0$ quantifying the range of the interaction.
\end{itemize}

The main result of the paper concerns the long-time behavior of the processes $(Z^i_t)_{t\geq 0}$.
 We investigate polynomial long-range interactions (see \eqref{eq:phialpha}) whose decay is tuned from short-range ($\mathbb{L}^2$ when $\alpha > 2$) and long-range (when $0<\alpha\leq 2$).
 The dynamics of the model depend on two parameters: the convolution kernel
$\varphi$
 which controls the temporal memory effect, and the exponent
$\alpha$
which controls the spatial decay of the interaction strength. In the subcritical case, the long-time behavior of the process~$Z^i$ depends on~$\varphi,$ $i$, and~$\alpha$. In the supercritical case, the behavior is the same as in the nearest-neighbor setting of \cite{delattre2016hawkes} and is independent of
$\alpha$. In particular, in the subcritical case, the range parameter~$\alpha$ influences the long-time behavior, but not in the supercritical case.

Note that, in the subcritical regime (i.e. when $\int_0^\infty \phi<1$), the proof is similar to that of \cite{delattre2016hawkes} and only requires adapting technical details.
In the supercritical regime (i.e. when $\int_0^\infty\phi>1$), in our setting, we cannot use their proofs relying on Feller's general renewal theorem (see \cite{feller197716}). Instead, we study the infinite-dimensional equation satisfied by the Laplace transforms of the processes $Z^i$ ($i\in\mathbb{Z}$), and use Tauberian theorems to deduce the long-time behavior of these processes. Our proofs also exploit properties of $\alpha$-stable laws.

This paper is structured as follows: in Section~\ref{sec2}, we introduce our model, which is an extension of the multivariate Hawkes process that incorporates long-range interactions between components, and  formulate the main results. Section~\ref{sec3} is dedicated to the proofs of the main results.
We investigate both subcritical and supercritical scenarios.
The technical details for these proofs and supplementary results are also given in the Appendix~\ref{appenda}.

\section{Main results}\label{sec2}
To state formally our model, let us introduce the following stochastic matrix indexed on $\mathbb{Z}$
\begin{equation}\label{eq:A}
A_\alpha(i,j) = \left\{\begin{array}{ll}\frac{c(\alpha)}{|i-j|^{1+\alpha}}&\textrm{if }i\neq j,\\0&\textrm{otherwise},\end{array}\right.
\end{equation}
where $c(\alpha)$ is a normalization constant.  The model that we consider in this paper is formally defined as the infinite-dimensional Hawkes process $(Z^i)_{i\in\mathbb{Z}}$ at \eqref{Hawkes} with intensities defined as
\begin{align*}
    \lambda^i_t =& \mu_i + \left(A_\alpha\cdot \left(\int_0^t \phi(t-s)dZ^j_s\right)_{j\in\mathbb{Z}}\right)^i = \mu_i + \sum_{j\neq i} \frac{c(\alpha)}{|i-j|^{1+\alpha}}\int_0^t \phi(t-s)dZ^j_s,
\end{align*}
where the kernel $\varphi$ is a non-negative measurable function, and $\mu_i$ ($i\in\mathbb{Z}$) are non-negative real numbers.
Whence, the processes $Z^i$ ($i\in\mathbb{Z}$) can be written as
\begin{equation}\label{Model}
Z_t^i = \int_0^t \int_0^{+\infty}
\mathbf{1}_{\left\{ z \;\le\; \mu_i + \sum_{j \neq i} A(i,j) \int_0^{s^-} \varphi(s-u) \, dZ_u^j \right\}}
\, \pi^i(ds,dz),
\end{equation}
where $\pi^i(dsdz)$ is a family of i.i.d. $(\mathcal{F}_t)_{t\geqslant 0}$-Poisson measures.

In all the paper, we work under the following assumptions.
\begin{assm}\label{hyp1}
    $ $
    \begin{itemize}
        \item the function~$\phi$ defined on $[0;+\infty)$ is measurable and non-negative,
        \item the sequence of non-negative numbers $(\mu_i)_{i\in\mathbb{Z}}$ is bounded.
    \end{itemize}
\end{assm}

Note that, under Assumption~\ref{hyp1}, the infinite-dimensional process $(Z^i)_i$ is well-posed by Theorem~6 of \cite{delattre2016hawkes} (see their Remark~5.(iii)). In addition, if we define
$$p_0=1~~\textrm{ and }~~p_i=1/|i|^{1+\alpha}~~(i\in\mathbb{Z}^*),$$
then, for all~$t\geq 0,$
$$\sum_{i\in\mathbb{Z}}p_i\mathbb{E}\left[Z^i_t\right]<\infty.$$

\subsection{Subcritical regime}
Let
\begin{equation}\label{I}
I=\int_0^{+\infty} \phi(t)\,dt.
\end{equation}
The subcritical regime is characterized by the condition $I\in(0;1).$ Since $A_\alpha$ is a stochastic matrix and $I<1$, we can define the following matrix
$$Q_\alpha^I(i,j)=\sum_{n\geq 0} I^n A_\alpha^n(i,j).$$

\begin{theorem}\label{thm:sub}
    Assume that $I<1$ and that Assumption~\ref{hyp1} holds true. Then, for any~$\alpha>0$,
    \begin{equation}
        \mathbb{E}\left|\frac{Z^ i_t}{t}-\sum_{j\in\mathbb{Z}}Q_\alpha^I(i,j)\mu_j\right|\underset{t\to\infty}{\longrightarrow}0.
    \end{equation}
\end{theorem}
We follow the same approach, adapting Lemma~27  of \cite{delattre2016hawkes} to our long-range interaction framework as proved in Appendix~\ref{appenda}(see Lemma~\ref{lem27}). We prove that the rate of convergence is the same for both short-range and long-range interactions, namely of order $t$, while the limiting value depends on the site
$i$ of the particle and on the parameter $\alpha$ that describes the interaction range.

\subsection{Supercritical regime}

In this section, we investigate the long-time behavior of our model in the supercritical regime (i.e. $I>1$). We develop the techniques used by Feller \cite{feller197716}. As claimed above, we study the Laplace transforms of the processes~$Z^i$ ($i\in\mathbb{Z}$).

Let us denote
$$\mathcal{L}_\phi(p) = \int_0^{+\infty} \phi(t) e^{-p\,t}dt.$$
\begin{theorem}\label{thm:sup2}
Under Assumption~\ref{hyp1}, suppose that $I>1$ and that there exist constants
$C>0$ and $\kappa>0$ such that
\begin{equation}\label{eq:exp}
\varphi(t) \le C e^{\kappa t};~~~~~~\textrm{for all }t\geq 0.
\end{equation}

Let $\theta>0$ be the unique solution to $\mathcal{L}_{\varphi}(\theta) = 1$,
and assume that~\eqref{eq:exp} holds true for some~$\kappa<\theta$ and $C>0$.\\
Define the spatial average
\[
\bar{\mu}
= \lim_{r \to \infty}
\frac{1}{\#\{ i \in \mathbb{Z} : |i| \le r \}}
\sum_{|i| \le r} \mu_i.
\]
Then, for each fixed $i \in \mathbb{Z}$,
\[
\mathbb{E}\left|
\frac{Z_t^i}{e^{\theta t}}
-
\frac{\bar{\mu}}{\theta^2
\int_0^{\infty} t \varphi(t) e^{-\theta t}\, dt}
\right|
\longrightarrow 0
\quad \text{as } t \to \infty.
\]
\end{theorem}

Note that, any function of the form,
$$\phi : t\in[0,+\infty)\longmapsto C\,e^{\lambda t}$$
satisfies the hypotheses of Theorem~\ref{thm:sup2}. As well as any function satisfying
$$\forall \lambda>0,~\phi(t)e^{-\lambda t}\underset{t\to\infty}{\longrightarrow}0.$$

\begin{remark}
In the supercritical case, unlike the subcritical regime, there is no distinction between long-range and short-range interactions. The long-time behavior of $Z^i$ is the same for all ranges independently of $\alpha$.
\end{remark}

\section{Proof of supercritical: Theorem~\ref{thm:sup2}}\label{sec3}

First, we require two technical lemmas. Lemma~\ref{lem27} is standard and its proof is deferred to the Appendix~\ref{appenda}. Lemma~\ref{lem16} below extends Lemma~16 of \cite{delattre2016hawkes} from the nearest-neighbor case to the long-range setting. Its proof exploits asymptotic properties of $\alpha$-stable laws in place of the central limit theorem.

\begin{lemma}\label{lem16}$ $
    \begin{itemize}
    \item[$(i)$] For any~$\alpha>0$,
   
    $$\underset{i\in\mathbb{Z}}{\sup}\sum_{j\in\mathbb{Z}} \left(A_\alpha^n(i,j)\right)^2\underset{n\to\infty}{\longrightarrow}0.$$
    \item[$(ii)$] For any $\alpha\in(0;2)$,
    $$\left(A_\alpha^n\mu\right)_i\underset{n\to\infty}{\longrightarrow}\bar\mu.$$
   
    \end{itemize}
\end{lemma}

\begin{proof}
(i) $\alpha>0:$ Firstly, we have $A_\alpha^n(i,j)=A_\alpha^n(0,j-i).$ Using that $A_\alpha$ is stochastic, we deduce that
\begin{equation}\label{16}
\varepsilon_n = \sum_{j\in\mathbb{Z}}(A_\alpha^n(i,j))^2 \leq  \sum_{j\in\mathbb{Z}}(A_\alpha^n(i,j)) \sup_{k\in\mathbb{Z}}(A_\alpha^n(0,k))\leq \sup_{k\in\mathbb{Z}}(A_\alpha^n(0,k)).
\end{equation}
            In addition, by Theorem~\ref{thm:tcllocal} and Lemma~\ref{lem:pn}, one has
\begin{eqnarray*}
A_\alpha^n(0,j)&\leq& p_n(j) + \left(A_\alpha^n(0,j) - p_n(j)\right) \\
                &\leq&\frac{1}{n^{1/\alpha}}p_1\left(j/n^{1/\alpha}\right)+o\left(n^{-1/\alpha}\right)\rightarrow 0\ as\ n\rightarrow\infty
\end{eqnarray*}
Therefore, $\varepsilon_n\rightarrow 0$  as $n\rightarrow\infty.$

(ii) Let $i\in\mathbb{Z}.$ We aim to prove that
\begin{equation}\label{D2i}
\lim_{n \to \infty}
\left(
(A_{\alpha}^n \bar{\mu})_i
-
\sum_{j \in \mathbb{Z}} p_n(j)\,\mu_j
\right)
= 0.
\end{equation}

Let us begin with the case~$\alpha>1$. For each $n \ge 1$, we choose $M_n$ satisfying
$n \ll M_n \ll n^{2/{\alpha}},$ (we can choose $M_n=n^{(1+2/{\alpha})/2}$).
 As  $(\mu_j)_{j \in \mathbb {Z}}$ is a bounded family, we have
\begin{align*}
\left| (A_{\alpha}^n \bar{\mu})_i - \sum_{j \in \mathbb{Z}} p_n(j)\,\mu_j \right|=& \left| \sum_{j \in \mathbb{Z}} A_{\alpha}^n(i,j)\, \mu_j - \sum_{j \in \mathbb{Z}} p_n(j)\, \mu_j \right|= \left| \sum_{j \in \mathbb{Z}} A_{\alpha}^n(0,j-i)\, \mu_j - \sum_{j \in \mathbb{Z}} p_n(j)\, \mu_j \right| \\
&\leq C \sum_{j \in \mathbb{Z}} \left| A_{\alpha}^n(0,j-i) - p_n(j) \right| \\
&\leq C \Bigg[
    \sum_{|j| \leq M_n} \left| A_{\alpha}^n(0,j-i) - p_n(j)\right|+ \sum_{|j| > M_n}  A_{\alpha}^n(0,j-i)+
    \sum_{|j| > M_n} p_n(j)\Bigg] \\
&\leq C \, [Y_1 + Y_2 + Y_3].
\end{align*}
For $Y_2,$ since
$$\sum_{j \in \mathbb{Z}} |j| \, A_{\alpha}^n(0,j)=\mathbb{E}[ \left| S_n \right|]\leq C n,$$ we have
\begin{align*}
Y_2=\sum_{|j| > M_n} A_{\alpha}^n(0,j-i)
&\leq \frac{1}{M_n} \sum_{j \in \mathbb{Z}} |j| \, A_{\alpha}^n(0,j-i)
= \frac{1}{M_n} \sum_{k \in \mathbb{Z}} |i+k| \, A_{\alpha}^n(0,k) \\
&\leq \frac{C}{M_n} (n + |i|) \;\longrightarrow 0 \quad \text{as } n \to \infty.
\end{align*}
Similarly, for $Y_3$,
$$Y_3=\sum_{|j| > M_n} p_n(j) \leq \frac{C}{M_n}n.$$
Finally, for $Y_1$,
\begin{align*}
    Y_1=&\sum_{|j| \leq M_n} \left| A_{\alpha}^n(0,j-i) - p_n(j)\right|\\
    \leq&\sum_{|j|\leq M_n} \left| A_{\alpha}^n(0,j-i)- p_n(j-i)\right|+\sum_{|j| \leq M_n} \left| p_n(j) - p_n(j-i)\right|\\\leq&\sum_{j\in \mathbb{Z}} \left| A_{\alpha}^n(0,j-i)- p_n(j-i)\right|+\sum_{|j|\leq M_n} \left| p_n(j) - p_n(j-i)\right|
\end{align*}
From Theorem \ref{thm:tcl},
$$\sum_{j\in \mathbb{Z}} \left| A_{\alpha}^n(0,j-i)- p_n(j-i)\right|$$ tends to $0$ as $n$ goes to $\infty,$ and from lemma \ref{lem:pn},
$$\left| p_n(j) - p_n(j-i) \right|
\leq \frac{1}{n^{1/\alpha}}
\left|
p_1\!\left(\frac{j}{n^{1/\alpha}}\right)
-
p_1\!\left(\frac{j-i}{n^{1/\alpha}}\right)
\right|\leq C \frac{|i|}{n^{2/\alpha}},$$
hence
$$\sum_{|j|\leq M_n} \left| p_n(j) - p_n(j-i)\right|\leq C|i|\frac{M_n}{n^{2/\alpha}},$$
tends to $0$ as $n$ goes to $\infty.$
Since $n\ll M_n\ll n^{2/\alpha},$ for example with the choice of $M_n=n^{(1+2/{\alpha})/2}$, we obtain the convergence stated in the lemma.

Now, let us treat the case $\alpha\leq 1$. The control of $Y_1$ is exactly the same. And, for $Y_2,Y_3$, it is no longer possible to consider $\mathbb{E}|S_n|$ since it is infinite. Instead, we consider $\mathbb{E}|S_n|^{\alpha-\eta}$ (for some small enough $\eta>0$) and the fact that $x\mapsto |x|^\beta$ is sub-additive for $0<\beta<1$. This gives
$$\mathbb{E}\left[\left|S_n\right|^{\alpha - \eta}\right]=\mathbb{E}\left[\left|\sum_{k=1}^n X_k\right|^{\alpha - \eta}\right]\leq \sum_{k=1}^n\mathbb{E}\left|X_k\right|^{\alpha - \eta} = C n.$$
Which gives a control for $Y_2$ and $Y_3$ of order $n/M_n^{\alpha - \eta}$. So, to obtain the convergence~\eqref{D2i}, $M_n$ has to satisfy $n^{1/(\alpha-\eta)}\ll M_n\ll n^{2/\alpha}$, where $\eta$ is small enough. So~\eqref{D2i} is proved. Then the statement~$(ii)$ follows from the same arguments as the ones of the end of the proof of Lemma~16 of \cite{delattre2016hawkes}.
\end{proof}

To study the asymptotics behavior of $\mathbb{E}{[Z^i_t]}$, we rely on Laplace transform and a Tauberian theorem.

 In the rest of this section, we assume that
\begin{equation}\label{eq:phi}
\phi(t) \leq C e^{\kappa t}.
\end{equation}
Let $\theta$ such that
$$\widehat{\phi} (\theta) = \int_0^\infty \phi(t) e^{-\theta t}dt=1,$$
which exists by the intermediate value theorem: $\widehat{\phi}(0)=I>1$ and $\widehat{\phi}(p)$ converges to zero as $p$ goes to infinity by~\eqref{eq:phi}. And we also assume that~\eqref{eq:phi} holds true for some $\kappa<\theta$.

Recall that
$$m^i_t = \mathbb{E}\left[Z^i_t\right]\textrm{ and }x^i_t = \mathbb{E}\left[\lambda^i_t\right].$$

Firstly, we prove that the Laplace transforms of the functions $x^i$ are well-defined and finite on~$(\theta;+\infty)$.
\begin{lemma}
For all $i\in\mathbb{Z}$, for any~$\beta>\theta,$
$$\int_0^{+\infty} x^i_t e^{-\beta t}dt<\infty.$$
\end{lemma}

\begin{proof}
    Start from
    $$x^i_t = \mu_i + \sum_{j\in\mathbb{Z}}A_\alpha(i,j)\int_0^t \phi(t-s)x^j_sds.$$
Let $\tilde x^i_t = x^i_t e^{-\beta t}$, and $\tilde \phi(t) = \phi(t)e^{-\beta t}$. Then
    $$\tilde x^i_t = \mu_i e^{-\beta t} + \sum_{j\in\mathbb{Z}}A_\alpha(i,j)\int_0^t \tilde \phi(t-s)\tilde x^j_sds.$$
    So,
    \begin{align*}
    \int_0^T \tilde x^i_t dt =& \frac{\mu_i}{\beta}\left(1-e^{-\beta t}\right) + \sum_{j\in\mathbb{Z}} A_\alpha(j,i)\int_0^T\int_0^t \tilde \phi(t-s)\tilde x^j_sds\\
    =&\frac{\mu_i}{\beta}\left(1-e^{-\beta t}\right) + \sum_{j\in\mathbb{Z}} A_\alpha(j,i)\int_0^T \tilde x^j_s\int_s^T \tilde \phi(t-s) dtds\\
    \leq& C + \rho\sum_{j\in\mathbb{Z}} A_\alpha(j,i)\int_0^t\tilde x^j_sds,
    \end{align*}
    with
    $$\rho = \int_0^\infty \tilde \phi(t)dt<1,$$
    since $\beta>\theta.$
Then, $F_i(t) = \int_0^t \tilde x^i_sds$ satisfies
    $$F_i(t) \leq C + \rho \sum_{j\in\mathbb{Z}} A_\alpha(j,i)F_j(t),$$
    so, with matrix notation (where the inequality holds for each term),
    $$F(t) \leq C + \rho A_\alpha\cdot F(t)$$
    and by induction (the iteration is legit since the coefficient of $A_\alpha$ are non-negative),
    $$F(t) \leq C\left[\sum_{k=0}^n \rho^k + \rho^{n+1} A^{n+1}_\alpha \cdot F(t)\right].$$
Then, since $A_\alpha$ is stochastic, $||A_\alpha||=1$, hence $||\rho A_\alpha|| = \rho <1$. So, letting $n$ go to infinity leads to
    $$F(t) \leq C\sum_{k=0}^{+\infty}\rho^k = \frac{C}{1-\rho}.$$
    This implies that, for all~$i,$ $F_i(t) = \int_0^t \tilde x^i_sds$ is bounded uniformly with respect to~$t$, which proves the result.
\end{proof}

Now, we can study the long-time behavior of the functions $m^i_t$ ($i\in\mathbb{Z}$).

\begin{lemma}\label{lem:equivm}
For all~$i\in\mathbb{Z},$
$$m^i_t \sim \frac{\bar\mu}{\theta^2\bar m}e^{\theta t},$$
where $\theta>0$ satisfies $\hat\phi(\theta)=1,$ $\bar \mu$ is defined in Theorem~\ref{thm:sup2}, and $\bar m = \int_0^\infty t\phi(t)e^{-\theta t}dt.$
\end{lemma}

\begin{proof}
Let
$$y^i_p = \widehat{x^i_t e^{-\theta t}}(p),$$
which is well-defined and finite for~$p>0$ thanks to the previous lemma.

  \begin{align*}
  x^i_t =& \mu_i + \sum_{j\in\mathbb{Z}}A_\alpha(i,j)\int_0^t \phi(t-s)x^j_sds\\
  x^i_t e^{-\theta t} =& \mu_ie^{-\theta t} + \sum_{j\in\mathbb{Z}}A_\alpha(i,j)\int_0^t \phi(t-s)x^j_se^{-\theta t}ds\\
   \bar{x}^i_t =& \mu_ie^{-\theta t} + \sum_{j\in\mathbb{Z}}A_\alpha(i,j)\int_0^t \bar\phi(t-s)\bar x^j_s ds,
   \end{align*}
   with $\bar\phi(t) = \phi(t)e^{-\theta t}$ and $\bar x^i_t = x^i_t e^{-\theta t}$. By standard properties of Laplace transform,
  $$y^i_p = \widehat{x^i_t e^{-\theta t}}(p)= \frac{\mu_i }{\theta +p}+\sum_{j\in\mathbb{Z}}A_\alpha(i,j)\widehat{\bar{\phi}}(p) y^j_p = \frac{\mu_i}{\theta + p} + \widehat{\bar\phi}(p) \left(A_\alpha \cdot y_p\right)_i.$$
  Then, we demonstrate by induction on~$N$ that
$$y^i_p = \frac{1}{\theta + p}\sum_{n=0}^{N} \widehat{\bar \phi}(p)^n \left(A^n_\alpha \mu\right)_i + \widehat{\bar\phi}(p)^{N+1} \left(A_\alpha^{N+1}\cdot y_p\right)_i$$
Then, we make~$N$ tend to infinity

\begin{align*}
    y^i_p =& \frac1{\theta + p}\sum_{n=0}^{+\infty} \widehat{\bar \phi}(p)^n \left(A^n_\alpha \mu\right)_i\\
    =& \frac{\bar\mu}{\theta + p}\sum_{n=0}^{+\infty} \widehat{\bar \phi}(p)^n + \frac1{\theta + p}\sum_{n=0}^{+\infty} \widehat{\bar \phi}(p)^n \left(\left(A^n_\alpha \mu\right)_i - \bar\mu\right)\\
    =& \frac{\bar\mu}{\theta+p}\cdot\frac{1}{1-\widehat{\bar \phi}(p)} + o\left(\frac{1}{1-\widehat{\bar \phi}(p)}\right),
\end{align*}
where Lemma~\ref{lem16}.(ii) is used at the last line. In addition,
$$\widehat{\bar \phi}(p) = 1- \bar m\, p + O(p^2),$$
with $\bar m = \int_0^\infty t\bar \phi(t)dt = -\widehat{\bar \phi}'(0).$ So
$$\frac{1}{1-\widehat{\bar \phi}(p)} = \left(\bar m\, p + O(p^2)\right)^{-1} = \frac1{\bar m\, p}\left(1 + O(p)\right)^{-1} = \frac{1}{\bar m\, p} + O(1).$$
Hence
$$y^i_p = \frac{\bar \mu}{\theta\,\bar m\, p} + o(1/p).$$
Whence, by classical Tauberian-type results (e.g. Haar's Tauberian theorem), this entails that
$$x^i_te^{-\theta t}\underset{t\to\infty}{\longrightarrow} \frac{\bar\mu}{\theta\,\bar m}.$$

$$x^i_t\sim \frac{\bar\mu}{\theta\,\bar m} e^{\theta t}$$
Finally, by classical analytical arguments about exponential functions,
$$m^i_t = \int_0^t x^i_s ds \sim \frac{\bar\mu}{\theta^2\bar m}e^{\theta t},$$
which ends the proof.
\end{proof}

Let us now conclude the proof of Theorem~\ref{thm:sup2}.
\begin{proof}[Proof of Theorem~\ref{thm:sup2}]
We use standard calculations of stochastic calculus and convolution equations. Recall that
$$m^i_t =\mathbb{E}[Z_t^i]= \mu_i t + \int_0^t\phi(t-s)\left(A_\alpha \cdot m_s\right)^ids.$$
Consider, for each $i \in \mathbb{Z}$, the martingale $$M^i_t = \int_0^t\int_0^{\infty}\textbf{1}_{\{z\leq\mu_i+\sum_{j\neq i}\frac{c(\alpha)}{\vert i-j\vert^{\alpha +1}}\int_0^{s-}\varphi(s-u)dZ_u^j\}}\tilde{\pi}^i(dsdz)$$
where $ \tilde{\pi}^i(dsdz)=\pi^i(dsdz)-dsdz.$ And let us define
\begin{align*}
U_t^i=& Z_t^i-m_t^i =M_t^i+\int_0^t\sum_{j\neq i}\frac{c(\alpha)}{\vert i-j\vert^{\alpha+1}}\varphi(t-s)(Z_s^j-m_s^j)ds\\=&M_t^i+\int_0^t \varphi(t-s)(A_\alpha \cdot U_s)^ids.
\end{align*}
This leads to
$$U_t^i = M_t^i +\sum_{n\geqslant 1}\int_0^t \varphi^{*n}(t-s){(A_\alpha^nM_s)}_ids,$$
hence
\begin{equation}\label{eqU}
\vert U_t^i\vert \leq \vert M_t^i\vert+\sum_{n\geqslant 1}\int_0^t \varphi^{*n}(t-s)\vert {(A_\alpha^n \cdot M_s)}^i\vert ds.
\end{equation}
We have
\begin{equation} \label{11}
\mathbb{E}\left[\vert M_t^i\vert\right]\leq \mathbb{E}\left[\left[\vert M_t^i,M_t^i\vert\right]\right]^{1/2}=\mathbb{E}\left[\left[Z_t^i\right]\right]^{1/2}=(m_t^i)^{1/2}\leq C e^{{\theta}t/2},
\end{equation}
where the last inequality comes from Lemma~\ref{lem:equivm}. In addition we can show that
\begin{equation}\label{11'}
\mathbb{E}[(A^n M^i_t)^2]=\sum_{j\in\mathbb{Z}}(A^n(i,j))^2\mathbb{E} {\vert M_t^j\vert}^2\leq \sum_{j\in\mathbb{Z}}(A^n(i,j))^2\mathbb{E} {\vert Z_t^j\vert}\leq Ce^{\theta t}\xi_n,
\end{equation}
where
$$\xi_n = \underset{i\in\mathbb{Z}}{\sup}\sum_{j\in\mathbb{Z}}(A^n(i,j))^2$$
vanishes by Lemma~\ref{lem16}.(i). Therefore, using equations~\eqref{11} and~\eqref{11'}, we can provide an alternative expression for equation (\ref{eqU}) as follows:
 $$\mathbb{E}\vert U_{t}^{i}\vert\leq C e^{-\theta t/2}+C e^{\theta t} \sum_{n\geq 1}\int_{0}^{t}\varphi^{*n}(t-s)\xi_{n}^{1/2}e^{{\theta}s/2}ds. $$

 Then by Lemma~\ref{lem26d},
 \begin{equation}\label{eqU0}
 \lim_{t\rightarrow\infty}\frac{\mathbb{E}\vert U_{t}^{i}\vert}{e^{\theta t}}=0.
 \end{equation}
Finally, from~\eqref{eqU0} and Lemma~\ref{lem:equivm},
 \begin{equation*}
\mathbb{E}\Big[ \Big\vert\frac{Z_t^i}{e^{\theta t}}-\frac{\bar\mu}{\theta^2 \bar m} \Big\vert\Big]\leq \frac{\mathbb{E}\vert U_{t}^{i}\vert}{e^{\theta t}}+\Big[\Big\vert\frac{m_t^i}{e^{\theta t}}-\frac{\bar\mu}{\theta^2 \bar m}\Big\vert\Big]
\end{equation*}
tends to $0$ as $t\to\infty.$ This completes the proof.
\end{proof}



\appendix

\section{\texorpdfstring{$\alpha$-stable distribution}{alpha-stable distribution}}\label{appenda}

Let us fix some $\alpha\in(0;2).$ Let $X_k$ ($k\in\mathbb{N}$) be i.i.d. random variables whose law is given by
$$\mathbb{P}\left(X_1=n\right) = \left\{\begin{array}{ll}\frac{c(\alpha)}{|n|^{1+\alpha}}&\textrm{if }n\in\mathbb{Z}^*,\\0&\textrm{otherwise},\end{array}\right.$$
and define $S_n$ the partial sum of the $X_k$ ($1\leq k\leq n$).

Let $Y_k$ ($k\in\mathbb{N}$) be i.i.d. random variables whose law is the symmetric stable law of parameter~$\alpha$: $S_\alpha(\sigma,0,0).$ Classically, the characteristic function of~$Y_1$ is
$$\Phi_\alpha:t\in\mathbb{R}\longmapsto e^{-|c\,t|^\alpha},$$
which is $\mathbb{L}^1$. Hence by Fourier inversion (see Theorem~$XV.3$ of \cite{feller1966}), the density function of $Y_1$ is
\begin{align*}
    p_1(x) =& \frac1{2\pi}\int_{-\infty}^{+\infty} e^{|c\,t|^\alpha + i\,t\,x}dt\\
    =&\frac1{2\pi}\int_{-\infty}^{+\infty} e^{|c\,t|^\alpha}\left(\cos(t\,x)+i\sin(t\,x)\right)dt\\
    =&\frac1\pi\int_0^{+\infty} e^{-c^{\alpha}\,t^\alpha}\cos(t\,x)dt.
\end{align*}
And, using classical stability properties of $\alpha$-stable laws, the density function of the partial sum of $Y_k$ ($1\leq k\leq n$) is
$$p_n(x)=\frac{1}{n^{1/\alpha}}p_1\left(\frac{x}{n^{1/\alpha}}\right).$$

The following lemma states properties of $p_1$ that are required in our proofs.
\begin{lemma}\label{lem:pn}
The function $p_1$ is bounded and Lipschitz continuous.
\end{lemma}

Next, by standard calculation, we have $\mathbb{P}(|X|>x) \sim c x^{-\alpha},$ so by Corollary XVII.5.2 of \cite{feller1966}, the law of $X_1$ belongs to the domain of attraction of a symetric stable law $S_\alpha(\sigma,0,0)$. Hence, there exist generalized central limit theorems to approximate the law of $S_n$ with law of density~$p_n$.

\begin{theorem}[Theorem~1.10 of \cite{szewczak2023classical}]\label{thm:tcllocal}
With the notation of this section,
    $$\underset{j\in\mathbb{Z}}{\sup}\left|\mathbb{P}\left(S_n = j\right) - p_n(j)\right| = o\left(n^{-1/\alpha}\right).$$
\end{theorem}

\begin{theorem}[Theorem~1.11 of \cite{szewczak2023classical}]\label{thm:tcl}
With the notation of this section,
    $$\sum_{j\in\mathbb{Z}}\left|\mathbb{P}\left(S_n = j\right) - p_n(j)\right|\underset{n\to\infty}{\longrightarrow}0.$$
\end{theorem}

\section{Technical lemmas}\label{appendb}

\begin{lemma}\label{lem27}
   Under Assumption~\ref{hyp1}, let $\alpha>0$, and recall that $A_\alpha$ is defined in~\eqref{eq:A}, and define
    $$p_0=1~~\textrm{ and }~~p_i=1/|i|^{1+\alpha}~~(i\in\mathbb{Z}^*).$$
    Then, the infinite-dimensional equation
    $$m^i_t = \mu_i t + \int_0^t \phi(t-s) \left(A_\alpha \cdot m_s\right)^ids$$
    has a unique solution satisfying $\sum_i p_i \sup_{s\leq t}m^i_s<\infty$ for all $t\geq 0$. The functions $m^i$ are $C^1$ and satisfies
    $$(m^i)'_t = \mu_i + \int_0^t \phi(t-s) (A_\alpha\cdot m'_s)^ids = \left(\left[I_d+\sum_{n\geq 1}A_\alpha^n\int_0^t\phi^{\star n}(s)ds\right]\mu\right)^i.$$
    And the function $u_t = \sup_i\sup_{s\leq t}(m^i_s)'$ is finite and satisfies $u_t\leq C + \int_0^t \phi(t-s)u_sds$.
\end{lemma}

\begin{proof}
    The only difference with the proof of Lemma~27 of \cite{delattre2016hawkes} concerns the Step~1 of their proof. We need to show here that, for any vector~$(x_i)_i$,
    $$\sum_{i\in\mathbb{Z}}p_i \left|(A_\alpha x)_i\right|\leq C_\alpha \sum_{i\in\mathbb{Z}} p_i|x_i|.$$
    This inequality comes directly from the fact that, for all~$j$,
    $$\sum_{i\neq j} p_i \frac{1}{|j-i|^{\alpha +1}}\leq C_\alpha p_j,$$
    whose proof follows directly from the computation of the proof of Remark~5.(iii) of \cite{delattre2016hawkes}.
\end{proof}

The next lemma is required to study the supercritical regime.

\begin{lemma}\label{lem26d}
Let $\phi:\mathbb{R}_+\to\mathbb{R}_+$ measurable. Assume that, for some $\kappa<\theta$ and~$C>0$, for any~$t\geq 0,$
$$\varphi(t) \le C e^{\kappa t}.$$
Then, for any vanishing sequence $(\eps_n)_n$
$$e^{-\theta t}\sum_{n=1}^{+\infty}\eps_n \int_0^t \phi^{\star n}(t-s) e^{\theta s/2}ds\underset{t\to\infty}{\longrightarrow}0.$$
\end{lemma}

\begin{proof}
    By induction, for all~$n\in\mathbb{N}^*$,
    $$\phi^{\star n}(t) \leq C^n \frac{t^{n-1}}{(n-1)!} e^{\kappa t}.$$
Let $\eta_n = \sup_{k\geq n} |\eps_k|$, which vanishes and is non-increasing, and write, for any~$N$,
    \begin{align*}
        &\underset{t\to\infty}{\limsup}~e^{-\theta t}\sum_{n=1}^{+\infty}|\eps_n| \int_0^t \phi^{\star n}(t-s) e^{\theta s/2}ds\\
        &=\underset{t\to\infty}{\limsup}~e^{-\theta t/2}\sum_{n=1}^{+\infty}|\eps_n| \int_0^t \phi^{\star n}(s) e^{-\theta s/2}ds\\
        &\leq  \underset{t\to\infty}{\limsup}~e^{-\theta t/2}\sum_{n=1}^{N}|\eps_n| \int_0^t \phi^{\star n}(s) e^{-\theta s/2}ds\\
        &~~~~+\underset{t\to\infty}{\limsup}~e^{-\theta t/2}\sum_{n=N+1}^{+\infty}|\eps_n| \int_0^t \phi^{\star n}(s) e^{-\theta s/2}ds\\
        &= S_1+S_2.
    \end{align*}
Since for every~$n$,
    $$e^{-\theta t/2} \int_0^t \phi^{\star n}(s)e^{-\theta s/2}ds\leq C^n e^{-\theta t/2}\int_0^t s^n e^{\kappa s}e^{-\theta s/2}ds.$$
Let $\delta>0$ small enough such that $\kappa + \delta<\theta$, and consider $C_\delta$ such that, for $t\geq 0,$ $s^n\leq  C_\delta e^{\delta s}$. Whence
    $$e^{-\theta t/2} \int_0^t \phi^{\star n}(s)e^{-\theta s/2}ds \leq C^n C_\delta e^{-\theta t/2} \int_0^t e^{(\kappa + \delta - \theta/2)s}ds \leq K e^{(\kappa + \delta - \theta)t}\underset{t\to\infty}{\longrightarrow}0,$$
    so $S_1=0$. To control $S_2$, let us notice that
    $$S_2 \leq \eta_N e^{-\theta t/2}\int_0^t \sum_{n=0}^{+\infty}\phi^{\star n}(s)e^{-\theta s/2}ds.$$
Denoting $F(t) = \sum_{n\geq 0} \phi^{\star n}(t)$, it is sufficient to prove that
    $$\underset{t\geq 0}{\sup}~~e^{-\theta t/2}\int_0^t F(s) e^{-\theta s/2}ds<\infty.$$
Let $U(t) = F(t) e^{-\theta t}$, then, by classical properties of Laplace transform, for any~$p>0,$
    $$\hat U(p) = \hat F(p + \theta) = \sum_{n\geq 0} \hat \phi(p+\theta)^n = \frac{1}{1-\hat \phi(p+\theta)} = \frac{1}{\bar m\, p} + O(1),$$
    since
    $$\hat \phi(p+\theta)= \hat\phi'(\theta) p + O(p^2) = -\bar m\,p + O(p^2).$$
Hence, by Tauberian-type results (see Feller version of Haar's Tauberian theorem, stated in the proof of Theorem~4 of \cite{feller197716}), $U(t)$ converges as $t$ goes to infinity, and so it is bounded. Implying that, for all~$t\geq 0$,
    $$F(t) \leq C e^{\theta t}.$$
    Whence
    $$e^{-\theta t/2}\int_0^t F(s) e^{-\theta s/2}ds\leq C e^{-\theta t/2}\int_0^t e^{\theta s/2}ds \leq \frac{C}{\theta}<\infty.$$
    This proves that $S_2$ vanishes as $N$ goes to infinity.
\end{proof}

\noindent
\textbf{Acknowledgements}\\
I would like to express my deep gratitude to my thesis supervisor, Arnaud Le Ny, for introducing me to this topic and for  his continuous support and insightful discussions.


\bibliography{reference}



\end{document}